\documentclass[reqno]{amsart}

\usepackage{mathrsfs}
\usepackage{amsmath}
\usepackage{amssymb}
\usepackage{cite}
\usepackage{latexsym}
\usepackage{graphicx}
\usepackage{amscd}

\usepackage{color}
\usepackage{comment}

\theoremstyle{plain}
\newtheorem{thm}{Theorem}[section]

\newtheorem{cor}[thm]{Corollary}
\newtheorem{dfn}[thm]{Definition}
\newtheorem{prop}[thm]{Proposition}
\newtheorem{rmk}[thm]{Remark}
\newtheorem{ex}[thm]{Example}

\def\D{\mathrm{D}}

\def\O{\mathscr{O}}
\def\T{\mathrm{T}}

\def\d{\mathrm{d}}

\def\s{\mathrm{s}}
\def\u{\mathrm{u}}

\def\Cset{\mathbb{C}}
\def\Nset{\mathbb{N}}
\def\Rset{\mathbb{R}}
\def\Sset{\mathbb{S}}
\def\Tset{\mathbb{T}}
\def\Zset{\mathbb{Z}}

\def\id{\mathrm{id}}

\def\epsilon{\varepsilon}

\makeatletter
 \@addtoreset{equation}{section}
\makeatother
\def\theequation{\arabic{section}.\arabic{equation}}

\begin{document}


\title[Smooth integrability of diffeomorphisms]%
{Smooth integrability of diffeomorphisms}

\author{Kazuyuki Yagasaki}

\address{Department of Applied Mathematics and Physics, Graduate School of Informatics,
Kyoto University, Yoshida-Honmachi, Sakyo-ku, Kyoto 606-8501, JAPAN}
\email{yagasaki@amp.i.kyoto-u.ac.jp}
\date{\today}

\begin{abstract}
Motivated by the notion of integrability introduced by Bogoyavlenskij for vector fields,
 we propose a definition of smooth integrability for general diffeomorphisms.
In brief, we say that a diffeomorphism is integrable
 if it commutes with the flows of {\color{black}commutative} vector fields and shares their first integrals.
We establish a Liouville-Arnold type theorem
 and prove that on connected invariant level sets of the first integrals,
 a smoothly integrable diffeomorphism is conjugate to a skew translation on a toroidal cylinder,
 and in particular, when the level sets are compact,
 the induced dynamics is quasiperiodic, as in the classical Hamiltonian case.
We further show that linear diffeomorphisms on real or complex Euclidean spaces
 are integrable in our sense,
 as is the case for linear vector fields in the sense of Bogoyavlenskij,
 and symplectic diffeomorphisms that are integrable in their standard Liouville-type definition
 are integrable in our framework.
We also modify the cotangent lift construction,
 which is useful in the treatment of integrability of vector fields,
 for diffeomorphisms,
 and show that integrability is preserved under this correspondence.
Several examples are provided to illustrate the scope of the theory.
 \end{abstract}
\maketitle


\section{Introduction}

Integrability {\color{black}and} solvability by quadrature of differential equations
 have attracted attention since before the days of Poincar\'{e} \cite{P92}.
For example, Bruns \cite{B87} discussed algebraic integrability of the classical three-body problem.
Especially recently, a vast literature on the topic has appeared,
 thanks to the celebrated Morales-Ramis theory \cite{M99,MR01,MRS07} via Ziglin's result \cite{Z82}.
The concept of integrability has also been developed for infinite-dimensional dynamical systems
 including partial differential equations, as {\color{black}is} well known \cite{A11,Z99}.
Thus, integrability has become a central theme in the study of dynamical systems.

For Hamiltonian systems, the standard definition of integrability is due to Liouville \cite{L55,A89}.
In particular, the Liouville-Arnold theorem (see, e.g., Sections {\color{black}49 and 50} of \cite{A89})
 guarantees that integrable Hamiltonian systems are solved by quadrature
 and their dynamics is very simple and typically quasiperiodic.
For general vector fields that may not be Hamiltonian,
 some definitions of integrability have been proposed (see, e.g., \cite{K13}).
Among them, the definition proposed by Bogoyavlenskij \cite{B98} is now widely used,
 especially in the study of nonintegrability of dynamical systems
 (see, e.g., \cite{Z02,C09,AZ10,ALMP18,MP20,MY18,AY20,Y22,Y23,MY23}),
 and seems the most natural
 since it includes the definition for Hamiltonian systems as a special case 
 and linear systems and linearizable ones
 are always integrable in this sense (see, e.g., \cite{Y25a}).

In contrast, for discrete dynamical systems,
which are typically given by diffeomorphisms,
integrability is often understood in a descriptive sense in the literature,
 referring to systems whose dynamics can be explicitly characterized,
 for instance through invariants, linearization, solvability,
{\color{black}Lax representations, or algebraic approaches
 (see, e.g., \cite{V91,KV02,S03,HJN16}).}
While a well-established definition exists for symplectic diffeomorphisms
 \cite{BRSG91,V91,HJN16}
 and a Liouville--Arnold type theorem has been proved in this setting,
 no widely accepted intrinsic definition of integrability is available
 for general diffeomorphisms, except {\color{black}in some restricted settings}
 \cite{CGM06,CGM08b,Z13,JS21}.
This situation naturally raises the question
 of whether a structural extension of Bogoyavlenskij-type integrability
 can be formulated for arbitrary diffeomorphisms.

In this paper, we introduce a notion of smooth integrability
 for general diffeomorphisms that are not necessarily symplectic,
 motivated by the notion of integrability introduced by Bogoyavlenskij \cite{B98}
 for vector fields.
In brief, we say that a diffeomorphism is integrable
 if it commutes with the flows of commuting vector fields and shares their first integrals.
{\color{black}
See Definition~\ref{dfn:2a} below for the precise definition.}
We establish a Liouville-Arnold type theorem
 and prove that on connected invariant level sets of the first integrals,
 a smoothly integrable diffeomorphism is conjugate to a skew translation on a toroidal cylinder,
 and in particular, when the level sets are compact,
 the induced dynamics is quasiperiodic, as in the classical Hamiltonian case.

We further investigate basic properties of smoothly integrable diffeomorphisms.
In particular, we show that linear diffeomorphisms on real or complex Euclidean spaces
 are always integrable in our sense,
 as is the case for linear vector fields in the sense of Bogoyavlenskij.
Moreover, symplectic diffeomorphisms that are integrable in their standard Liouville-type definition
 are proved to be integrable in our framework.
We also modify the cotangent lift construction for diffeomorphisms
 and show that integrability is preserved under this correspondence.
{\color{black}
Thus, although Definition~\ref{dfn:2a} may appear restrictive at first sight,
 it covers several natural and important classes of diffeomorphisms.}
Several examples, including the multi-dimensional Lyness map
 \cite{KLR93,CGM07,CGM08a,CGM08b,BR09}, 
 are provided to illustrate the scope of the theory.

{\color{black}
On the other hand, Lax representations \cite{RVKQ07,S03,SV03}
 and algebraic approaches \cite{BV99,KV02}
 provide important frameworks for integrable maps.
The former often yield first integrals through spectral invariants of Lax matrices,
 whereas the latter include approaches based on additional algebraic structures,
 as well as criteria based on degree growth or algebraic entropy for rational
 or birational maps.
These approaches are therefore different in nature from the notion of integrability considered here.}

The outline of this paper is as follows:
In Section~2, we briefly review the integrability of vector fields
 in the sense of Liouville and Bogoyavlenskij,
 including their definitions and relationship.
In Section~3, we consider general diffeomorphisms
 and provide our definition of their integrability
 and some basic properties of integrable diffeomorphisms,
 including their relationship with the flows of integrable vector fields
 and the integrability of linear diffeomorphisms.
We generalize the Liouville-Arnold theorem to diffeomorphisms
 and show that the dynamics of integrable diffeomorphisms are very regular,
 like those of integrable vector fields in Section~4.
In addition, we discuss the non-integrability of such diffeomorphisms
 as they exhibit chaotic dynamics.
Finally, we consider symplectic diffeomorphisms
 and recall the standard definition of their integrability in Section~5.
We discuss its relationship {\color{black}with integrability of general diffeomorphisms,
 including their cotangent lifts.}


\section{Integrability of Vector Fields}
We first recall the integrability of vector fields.
We begin with the Hamiltonian case.

Let $(M^{2n},\Omega)$ be a $2n$-dimensional $C^\infty$ or analytic symplectic manifold
 for $n\in\Nset$, i.e., $M^{2n}$ is a $2n$-dimensional $C^\infty$ or analytic manifold
 and $\Omega$ is a closed nondegenerate differential two-form on it.
We consider an $n$-degree-of-freedom $C^r$ Hamiltonian system 
\begin{equation}
\dot{x}=X_H(x),\quad
x\in M^{2n},
\label{eqn:Hsys}
\end{equation}
having a $C^{r+1}$ Hamiltonian function $H:M^{2n}\to\Rset$ on $(M^{2n},\Omega)$,
 where $X_H$ denotes the Hamiltonian vector field for $H$,
 i.e., the interior product of $\Omega$ and $X_H$ satisfies
\[
i_{X_H}\Omega=\d H.
\]
Here and hereafter $r\ge 1$, and $r=\infty$ and $\omega$,
 for the latter of which $H$ is analytic, are allowed.
We take $r+1=\infty$ and $\omega$ when $r=\infty$ and $\omega$, respectively.
We also assume that $M^{2n}$ is analytic when $r=\omega$.
By Darboux's theorem (see, e.g., Section~43 of \cite{A89}),
 there is a local coordinate system $(q,p)\in\Rset^n\times\Rset^n$
 such that the symplectic form is represented by
\[
\Omega=\sum_{j=1}^n\d p_j\wedge\d q_j,
\]
where $q_j$ and $p_j$ denote the $j$-th elements of $q$ and $p$, respectively.
In the coordinate system, the Hamiltonian system \eqref{eqn:Hsys} becomes
\begin{equation}
\dot{q}=\D_p H(q,p),\quad
\dot{p}=-\D_q H(q,p),
\label{eqn:Hsys1}
\end{equation}
where `$\D_q$' and `$\D_p$' represent the partial differentiation with respect to $q$ and $p$, respectively.
The definition of integrability in the sense of Liouville \cite{L55,A89}
 is stated  for \eqref{eqn:Hsys} as follows.

\begin{dfn}[Liouville]
\label{dfn:1a}
The Hamiltonian system~\eqref{eqn:Hsys} is called \emph{integrable}
 if {\color{black}there exist} $n$ scalar-valued functions $F_1(:=H),F_2,\dots,F_{n}$ such that
 the one-forms $\d F_1,\dots, \d F_{n}$
 are linearly independent almost everywhere $($a.e.$)$
 and in involution, i.e., their \emph{Poisson brackets} satisfy
\[
\{F_j,F_k\}(x):=\Omega(X_{F_j}(x),X_{F_k}(x))\equiv 0
\]
for $j,k\in[n]:=\{1,2,\ldots,n\}$.
We say that the Hamiltonian system \eqref{eqn:Hsys}
 is \emph{analytically} $($resp. ${\color{black}C^r})$ \emph{integrable}
 if the first integrals are analytic $($resp. ${\color{black}C^{r+1}})$. 
\end{dfn}

When $n=1$, the Hamiltonian system \eqref{eqn:Hsys} is $C^r$ integrable.
{\color{black}In general,} if it is $C^r$ integrable in the sense of Liouville {\color{black}with $r\ge 2$},
 the level set $F^{-1}(c)$ with $F(x)=(F_1(x),\ldots,F_n(x))$ and $c\in\Rset^n$ is connected
 and $\d F(x)$ is nondegenerate on it,
 then the Hamiltonian system \eqref{eqn:Hsys} can be solved by quadrature
 and transformed into
\begin{align}
&
\dot{I}_j=0,\quad
\dot{\theta}_j=\frac{\partial H}{\partial I_j}(I),\quad
j\in[n],\notag\\
&
I=(I_1,\ldots,I_n)\in U_I,\quad
\theta=(\theta_1,\ldots,\theta_n)\in\Tset^\ell\times\Rset^{n-\ell},
\label{eqn:aasys}
\end{align}
{\color{black}by a $C^{r}$ change of coordinates,}
where {\color{black}$H(I)$ is $C^{r+1}$,}
 $U_I$ is a neighborhood of $I=0$ in $\Rset^n$ and $\ell\in[n]\cup\{0\}$.
See, e.g., Sections~49 and 50 of \cite{A89}.
This result is referred to as the Liouville-Arnold or Arnold-Jost theorem.
In particular, if $\ell=n$, then the level set $F^{-1}(c)$
 is diffeomorphic to the $n$-dimensional torus $\Tset^n$
 and the Hamiltonian system \eqref{eqn:Hsys} exhibits quasiperiodic motion on it,
 as in the standard statement of the Liouville-Arnold theorem. 
Bruns \cite{B87} showed that
 the classical three-body problem is not algebraically integrable in the sense of Liouville.
His result has been extended in \cite{P92,T01,BW03,Y24} subsequently.

We turn to the general case.
Let $M^n$ be an $n$-dimensional $C^\infty$ or analytic manifold.
Consider $n$-dimensional dynamical systems of the form
\begin{equation}
\dot{x}=X(x),\quad x\in M^n,
\label{eqn:gsys}
\end{equation}
where $n\in\Nset$ and $X:M^n\to TM^n$ is $C^r$.
We assume that $M^n$ is analytic when $r=\omega$ hereafter like $M^{2n}$.
The definition of integrability in the sense of Bogoyavlenskij \cite{B98}
 is stated for \eqref{eqn:gsys} as follows.

\begin{dfn}[Bogoyavlenskij]
\label{dfn:1b}
For $m\in\Nset$ such that $1\le m\le n$,
 the system~\eqref{eqn:gsys} is called \emph{$(m,n-m)$-integrable} or simply \emph{integrable}
 if there exist $m$ vector fields $X_1(:=X),X_2,\dots,X_m$
 and $n-m$ scalar-valued functions $F_1,\dots,F_{n-m}$ on $M^n$ such that
 the following two conditions hold$\,:$
\begin{enumerate}
\setlength{\leftskip}{-1.8em}
\item[\rm(i)]
$X_1,\dots,X_m$ are linearly independent a.e.
 and commute with each other, i.e.,
\[
[X_j,X_k](x):=\left(X_jX_k-X_kX_j\right)(x)\equiv 0\quad\mbox{for $j,k\in[m]$},
\]
where $[\cdot,\cdot]$ denotes the Lie bracket$;$
\item[\rm(ii)]
The one-forms $\d F_1,\dots, \d F_{n-m}$ are linearly independent a.e.
 and $F_1,\dots,F_{n-m}$ are first integrals of $X_1, \dots,X_m$, i.e.,
\[
X_j(F_k)(x):=\d F_k(x)(X_j(x))\equiv 0\quad\mbox{for $j\in[m]$ and $k\in[n-m]$}.
\]
\end{enumerate}
We say that the system \eqref{eqn:gsys}
 is \emph{analytically} $($resp. ${\color{black}C^r})$ \emph{integrable}
 if the first integrals and commutative vector fields are analytic $($resp. ${\color{black}C^r})$.
\end{dfn}

In a local coordinate system on $M^n$,
 if $X_j(x)=\sum_{k=1}^n v_{jk}(x)\partial/\partial x_k$
 and $v_j(x)=(v_{j1}(x),\ldots,\linebreak v_{jn}(x))$ for $j\in[m]$,
 then conditions~(i) and (ii) are stated as follows:
\begin{enumerate}
\setlength{\leftskip}{-1.8em}
\item[\rm(i)]
$v_1(x),\dots,v_m(x)$ are linearly independent a.e.
 and commute with each other, i.e.,
\[
[v_j,v_k](x):=\D v_k(x)v_j(x)-\D v_j(x)v_k(x)\equiv 0\quad\mbox{for $j,k\in[m]$},
\]
where $\D$ represents the partial differentiation with respect to $x\,;$
\item[\rm(ii)]
The derivatives $\D F_1(x),\dots, \D F_{n-m}(x)$ are linearly independent a.e.
 and $F_1(x),\linebreak\dots,F_{n-m}(x)$ are first integrals of $v_1(x), \dots,v_m(x)$, i.e.,
\[
\D F_k(x)\cdot v_j(x)\equiv 0\quad\mbox{for $j\in[m]$ and $k\in[n-m]$},
\]
{\color{black}where `$\cdot$' denotes the standard inner product.}
\end{enumerate}

If it is integrable in the sense of Bogoyavlenskij,
 the level set $F^{-1}(c)$ with $F(x)=(F_1(x),\ldots,F_{n-m}(x))$ and $c\in\Rset^{n-m}$
 is connected and the $m$ commutative vector fields are linearly independent on it, 
 then the system \eqref{eqn:gsys} can be transformed
 {\color{black}by a $C^r$ change of coordinates} into
\begin{align}
&
\dot{I}_j=0,\quad
\dot{\theta}_j=\omega_j(I),\quad
j\in[n-m],\notag\\
&
I=(I_1,\ldots,I_{n-m})\in U_I,\quad
\theta=(\theta_1,\ldots,\theta_m)\in\Tset^\ell\times\Rset^{m-\ell},
\label{eqn:gaasys}
\end{align}
which is similar to \eqref{eqn:aasys}
 but $U_I$ is a neighborhood of $I=0$ in $\Rset^{n-m}$ and $\ell\in[m]\cup\{0\}$,
 where $\omega_j:U_I\to\Rset$, $j\in[m]$, are $C^r$.
See, e.g., \cite{B98} and the proof of Theorem~\ref{thm:3a} below.
When $n=1$, the system \eqref{eqn:gsys}
 is $C^r$ $(1,0)$-integrable immediately.
Definition~\ref{dfn:1b} is regarded as a generalization
 of Definition~\ref{dfn:1a} for Hamiltonian systems
 since an $n$ degree-of-freedom Liouville integrable Hamiltonian system with $n\ge 1$
 has not only $n$ functionally independent first integrals $F_j(x)$, $j\in[n]$,
 but also $n$ linearly independent commutative (Hamiltonian) vector fields $X_{F_j}$, $j\in[n]$.
In particular, linear vector fields are always analytically $(n,0)$-integrable in this sense
 (see, e.g., Proposition~3.13 of \cite{Y25a}).
Moreover, it includes the primitive one when $m=1$
 and the Lie case when $m=n$ (see, e.g., \cite{K13}).
It was also shown in \cite{Y25a} that the {\color{black}linearizability} of vector fields
 is closely related to their {\color{black}integrability}. 

On the other hand,
 we can regard the system \eqref{eqn:gsys} as a part of a Hamiltonian system
 by \emph{cotangent lift}.
Let $T^\ast M^n$ denote the cotangent bundle of $M^n$
 and let $(x,p)\in\Rset^n\times\Rset^n$ represent a local coordinate system there.
A Hamiltonian system with the Hamiltonian $H(x,p)=p\cdot v(x)$ is given by
\begin{equation}
\dot{x}=v(x),\quad
\dot{p}=-\D v(x)^\T p
\label{eqn:Hsys2}
\end{equation}
in the local coordinate system,
 where $v(x)$ represents the vector field of \eqref{eqn:gsys}
 and the superscript `${}\T$' represents the transpose operator.
If $\tilde{v}(x)$ is a commutative vector field for \eqref{eqn:gsys},
 then $\tilde{F}(x,p)=p\cdot\tilde{v}(x)$ is a first integral since
\begin{align*}
&
\D_x \tilde{F}(x,p)\cdot v(x)-\D_p \tilde{F}(x,p)\cdot\D v(x)^\T p\\
&
=(\D \tilde{v}(x)^\T p)\cdot v(x)-\tilde{v}(x)\cdot\D v(x)^\T p\\
&=p\cdot(\D \tilde{v}(x)v(x)-\D v(x)\tilde{v}(x))\equiv 0.
\end{align*}
Thus, we prove the following as in Proposition~2 of \cite{AZ10}.

\begin{prop}
\label{prop:1a}
If the system~\eqref{eqn:gsys} is Bogoyavlenskij integrable,
 then the Hamiltonian system~\eqref{eqn:Hsys2} is Liouville integrable.
\end{prop}

This result plays an important role in the extension of the Morales-Ramis theory \cite{M99,MR01,MRS07}
 to non-Hamiltonian systems \cite{AZ10}.
 

\section{Definition of Integrable Diffeomorphisms}

We turn to general diffeomorphisms and define their integrability.
Let $f:M^n\to M^n$ be a $C^r$ diffeomorphism,
 where $M^n$ is a $C^\infty$ or analytic manifold, as in Section 2‌.
We propose the following definition of its integrability.

\begin{dfn}[Integrability of diffeomorphisms]
\label{dfn:2a}
For $m\in\Nset$ with $0\le m\le n$,
 $f$ is called \emph{$(m,n-m)$-integrable} or simply \emph{integrable}
 if there exist $m$ vector fields $X_1,\ldots,X_m$
 and $n-m$ scalar-valued functions $F_1,\ldots,F_{n-m}$ such that
 the following three conditions hold$\,:$
\begin{enumerate}
\setlength{\leftskip}{-1.6em}
\item[\rm(i)]
$X_1,\dots,X_m$ are linearly independent a.e.
 and commute with each other$\,;$
\item[\rm(ii)]
$\d F_1,\ldots,\d F_{n-m}$ are linearly independent a.e.
 and $F_1,\ldots,F_{n-m}$ are first integrals of $X_1,\dots,X_m;$
\item[\rm(iii)]
The flow $\varphi_j^{t}$ of $X_j$ commutes with $f$, i.e.,
\begin{equation}
f\circ\varphi_j^{t}(x)-\varphi_j^{t}\circ f(x)\equiv 0,
\label{eqn:dfn2a}
\end{equation}
for any $t\in\Rset$ and $j\in[m]$, and $F_j$ is also a first integral of $f$, i.e.,
\[
F_j(f(x))\equiv F_j(x),
\]
for $j\in[n-m]$.
\end{enumerate}
We say that $f$ is \emph{analytically} $($resp. ${\color{black}C^r})$ \emph{integrable}
 if the first integrals and commutative vector fields as well as $f$ are analytic $($resp. ${\color{black}C^r})$.
\end{dfn}

\begin{rmk}\
\label{rmk:2a}
\begin{enumerate}
\setlength{\leftskip}{-1.6em}
\item[\rm(i)]
In {\rm\cite{CGM06}}, an $n$-dimensional diffeomorphism was called \emph{integrable}
 if it has $n$ functionally independent first integrals,
 equivalently $m=0$ in Definition~{\rm\ref{dfn:2a}},
 in which the existence of commutative vector fields are not required. 
Moreover, it was shown that if $f:U\to U$ is integrable in this meaning and
\[
\mathrm{card}\left(\bigcap_{j=1}^n\{F_j(x)-c_j\}\cap U\right)\le k
\quad\mbox{for any $c_j\in\Rset$, $j\in[n]$,}
\]
for some $k\in\Nset$,
 where $U$ is an open set in $\Rset^n$ and  $F_j$, $j\in[n]$, are differentiable first integrals,
 then there exists $\ell\in\Nset$ such that $f^\ell(x)=x$ for any $x\in U$.
\item[\rm(ii)]
In {\rm\cite{CGM08b,Z13}},
 an $n$-dimensional diffeomorphism was called \emph{integrable}
 if it has $n-1$ functionally independent first integrals.
This condition is weaker than those of Definition~$\ref{dfn:2a}$.
In {\rm\cite{CGM08b}}, it was shown that
 such a diffeomorphism has an orbit exhibiting translation under one of additional conditions,
 including commutativity with the flow of a vector field.
In {\rm\cite{Z13}}, analytic diffeomorphisms of the form $Ax+O(|x|^2)$,
 which have a fixed point at $x=0$,
 where $A$ is an $n\times n$ matrix having no eigenvalue with modulus one,
 were considered and proven to be analytically conjugate
 to the flows of analytically $(1,n-1)$-integrable vector fields near $x=0$ in the sense of Bogoyavlenskij,
 so that they are analytically $(1,n-1)$-integrable near $x=0$ in the sense of Definition~$\ref{dfn:2a}$.
Thus, one needs more than the existence of $n-1$ first integrals
 for {\color{black}a more suitable definition of} integrability for $n$-dimensional diffeomorphisms. 
\item[\rm(iii)]
In {\rm\cite{JS21}}, $n$-dimensional smooth diffeomorphisms
 of the form $Ax+O(|x|^2)$ with an $n\times n$ matrix $A$
 were also considered and called \emph{integrable}
 if they have $m$ commutative diffeomorphisms of $A_jx+O(|x|^2)$, $j\in[m]$,
 including themselves, and $n-m$ common first integrals.
Normal forms of such diffeomorphisms near the origin were discussed
 when the $n\times n$ matrices $A_j$, $j\in[m]$, are diagonalizable, the first integrals are homogeneous
 and the commutative diffeomorphisms are the time-one maps of certain vector fields
 or $n$-dimensional vectors
 whose elements are the logarithms of moduli of the eigenvalues of $A_j$, $j\in[m]$,
 are linearly independent.
Thus, a simple replacement of commutative vector fields with commutative diffeomorphisms
 in Definition~$\ref{dfn:2a}$ is not appropriate for the integrability of diffeomorphisms. 
\end{enumerate}
\end{rmk}

{\color{black}
The following proposition gives an equivalent infinitesimal form of the relation \eqref{eqn:dfn2a}.

\begin{prop}
\label{prop:2d}
$X$ is a commutative vector field of a diffeomorphism $f$ if and only if
\begin{equation}
\d f(x)X(x)\equiv X(f(x)).
\label{eqn:prop2d1}
\end{equation}
\end{prop}

\begin{proof}
Suppose that $X$ is a commutative vector field of $f$.
Let $\varphi^t$ denote the flow of $X$.
Then
\begin{equation}
f(\varphi^t(x))-\varphi^t(f(x))\equiv 0.
\label{eqn:prop2d2}
\end{equation}
Differentiating the above relation with respect to $t$ at $t=0$, we have
\[
\d f(x)X(x)-X(f(x))\equiv 0,
\]
which yields \eqref{eqn:prop2d1}.

Conversely, suppose that Eq.~\eqref{eqn:prop2d1} holds.
We have
\[
\frac{\d}{\d t}f(\varphi^t(x))=\d f(\varphi^t(x))X(\varphi^t(x))=X(f(\varphi^t(x)))
\]
and
\[
\frac{\d}{\d t}\varphi^t(f(x))=X(\varphi^t(f(x)))
\]
for any $x\in M^n$.
Hence, $f(\varphi^t(x))$ and $\varphi^t(f(x))$ satisfy the same differential equation
\[
\dot{x}=X(x).
\]
Since $f(\varphi^t(x)),\varphi^t(f(x))=f(x)$ at $t=0$,
 we obtain \eqref{eqn:prop2d2}
  by the uniqueness of solutions to the initial value problem of differential equations.
\end{proof}

\begin{rmk}
A vector field $X$ satisfying \eqref{eqn:prop2d1}
 is often called a Lie symmetry of $f$ $($see, e.g., {\rm\cite{HBQC96,CGM08b})}.
It is also referred to as a continuous symmetry in {\rm\cite{DLM12}}
and as a symmetry field in {\rm\cite{ST16}}.
\end{rmk}
}

We show {\color{black}below} that Definition~\ref{dfn:2a} does not give an unnecessary restriction 
 as integrable diffeomorphisms,
 compared with Definition~\ref{dfn:1b} for integrable vector fields.
We immediately obtain the following.

\begin{prop}
\label{prop:2a}
Let $\varphi^t$ be the flow generated by \eqref{eqn:gsys}.
If the system \eqref{eqn:gsys} is $(m,n-m)$-integrable in the sense of Bogoyavlenskij,
 then the flow $\varphi^t$ is $(m,n-m)$-integrable in the sense of Definition~$\ref{dfn:2a}$
 for any $t\in\Rset$.
\end{prop}

\begin{proof}
If the system \eqref{eqn:gsys} is $(m,n-m)$-integrable in the sense of Bogoyavlenskij,
 then its $m$ commutative vector fields and $n-m$ first integrals
 satisfy conditions~(i)-(iii) in Definition~$\ref{dfn:2a}$ for the flow $\varphi^t$ at any $t\in\Rset$.
\end{proof}

We now give three simple examples {\color{black}of one-dimensional diffeomorphisms.}

\begin{ex}
Consider the one-dimensional diffeomorphism
\begin{equation}
f(x)=ax+b,\quad
x\in\Rset,
\label{eqn:ex2a1}
\end{equation}
where $a\neq 0$ and $b\in\Rset$ are constants.
If $a=1$, then $v_1(x)=1$ is a vector field such that its flow satisfies
\begin{equation}
f(\varphi_1^{t}(x))-\varphi_1^{t}(f(x))=0\quad\mbox{for any $t\in\Rset$}
\label{eqn:ex2a2}
\end{equation}
since $\varphi_1^t(x)=x+t$.

We now assume that $a\neq 1$.
Suppose that $v_1(x)=\alpha x+\beta$ with $\alpha\neq 0$ is a vector field
 such that the flow $\varphi_1^t(x)$ of \eqref{eqn:ex2a1} satisfies \eqref{eqn:ex2a2}.
Since $\varphi_1^t(x)=(x+\beta/\alpha)e^{\alpha t}-\beta/\alpha$,
 Eq.~\eqref{eqn:ex2a2} becomes
\[
a((x+\beta/\alpha)e^{\alpha t}-\beta/\alpha)+b
=(ax+b+\beta/\alpha)e^{\alpha t}-\beta/\alpha,
\]
which yields
\[
\beta/\alpha=b/(a-1).
\]
Hence, if we take $\alpha,\beta$ such that the above conditions holds,
 then Eq.~\eqref{eqn:ex2a2} holds.
Thus, the one-dimensional diffeomorphism \eqref{eqn:ex2a1} is $(1,0)$-integrable.
\end{ex}

\begin{ex}
\label{ex:2b}
Let $a\in\Sset^1$ and consider the rigid rotation
\begin{equation}
f(x)=x+a,\quad
x\in\Sset^1,
\label{eqn:ex2b}
\end{equation}
which is a diffeomorphism on $\Sset^1{\color{black}:=\Rset/2\pi\Zset}$.
Obviously, the vector field $v_1(x)=1$ satisfies \eqref{eqn:ex2a2} for any $a\in\Sset^1$.
Hence, the rigid rotation \eqref{eqn:ex2b} is $(1,0)$-integrable.
\end{ex}

\begin{ex}
\label{ex:2c}
Consider a one-dimensional analytic diffeomorphism {\color{black}on $\Sset^1$}
 given by
\begin{equation}
f(x)=x+\pi/k+\epsilon\sin^2kx,\quad
x\in\Sset^1,
\label{eqn:ex2c1}
\end{equation}
where $k\in\Nset$ and $0<\epsilon<1/k$.
The diffeomorphism \eqref{eqn:ex2c1} was discussed in Warning~$1.6$ of {\rm\cite{M84}}. 
In particular, as shown easily,
 $f([(j-1)\pi/n,j\pi/n])=[j\pi/n,(j+1)\pi/n]$,
 $\{(j-1)\pi/n\}_{j=1}^{n}$ is a periodic orbit of period $2n$,
 and $\{f^j(x)-j\pi/n\}_{j=0}^\infty$ is an increasing sequence with upper bound $\pi/n$ 
 when $0<x<\pi/n$.
The last implies that the diffeomorphism \eqref{eqn:ex2c1} has no $C^1$ first integral
 since it has the same value at the infinitely many points $f^j(x)$, $j\in\Nset$,
 for any $x\in(0,\pi/n)$ otherwise.
 
Suppose that $v_1(x)$ is a $C^1$ vector field on $\Sset^1$
 such that $v_1(x)\neq 0$ a.e.
 and  its flow $\varphi_1^t(x)$ satisfies \eqref{eqn:ex2a2}.
{\color{black}By Proposition~$\ref{prop:2d}$}, we have
\begin{equation}
f'(x)v_1(x)=v_1(f(x)),
\label{eqn:ex2c2}
\end{equation}
where the prime represents differentiation with respect to $x$.

Let $x=x^\ast$ be an equilibrium of $v_1(x)$.
Since $f'(x)>0$ on $\Sset^1$,
 $x=f^j(x^\ast)$ is also an equilibrium of $v_1(x)$ for any $j\in\Zset$ by \eqref{eqn:ex2c2}.
Hence, we easily see that $x^\ast=(j-1)\pi/n$ for some $j\in[2n]$.
Indeed, if not, $v_1(f^j(x))=0$, $j\in\Zset$, for some $x\in(0,\pi/n)$,
 so that $v_1(x)=0$ near $x=(j-1)\pi/n$, $j\in[2n]$, and consequently $v_1(x)\equiv 0$.
It follows from \eqref{eqn:ex2c2}
 that $v_1(x)$ has the same sign except at $x=(j-1)\pi/n$, $j\in[2n]$,
 since $f'(x)>0$ on $\Sset^1$ and $f([(j-1)\pi/n,j\pi/n])=[j\pi/n,(j+1)\pi/n]$ for $j\in[2n]$.
Without loss of generality, we assume that $v_1(x)\ge 0$.
So $v_1(x)$ has a maximum, say at $x=x^{\ast\ast}$, on $[0,\pi/n]$.
Using \eqref{eqn:ex2c2} repeatedly, we obtain
\[
v_1(f^{2n}(x^{\ast\ast}))=v_1(x^{\ast\ast})\frac{\d}{\d x}(f^{2n}(x^{\ast\ast})).
\]
Since $\{f^j(x^{\ast\ast})-j\pi/n\}_{j=0}^\infty$ is an increasing sequence,
 we conclude that $v_1(f^{2n}(x^{\ast\ast}))>v_1(x^{\ast\ast})$.
This contradicts that $x^{\ast\ast}$ is a maximum of $v_1(x)$ on $[0,\pi/n]$,
 since $f^{2n}(x^{\ast\ast})\in[0,\pi/n]\mod 2\pi$.
Thus, $v_1(x)$ has no equilibrium.

Integrating \eqref{eqn:ex2c2}, we have
\begin{equation}
g(f(x))=g(x)+c
\label{eqn:ex2c3}
\end{equation}
for $x\neq f^{j-1}(0)$, $j\in[n]$, where $c$ is a constant and
\[
g(x)=\int_0^x\frac{\d y}{v_1(y)}.
\]
Using \eqref{eqn:ex2c3} repeatedly, we obtain
\begin{equation}
g(f^{2n}(x))=g(x)+2nc,
\label{eqn:ex2c4}
\end{equation}
which yields $g(2\pi)=2nc$ at $x=0$.
On the other hand, for $x\in(0,\pi/n)$, we have
\begin{equation}
g(f^{2n}(x))-g(x)=g(f^{2n}(x))-g(x+2\pi)+g(x+2\pi)-g(x)>2nc
\label{eqn:ex2c5}
\end{equation}
since $f^{2n}(x)>x+2\pi$ and by the periodicity of $v_1(x)$
\[
g(x+2\pi)-g(x)=\int_x^{x+2\pi}\frac{\d y}{v_1(y)}=\int_0^{2\pi}\frac{\d y}{v_1(y)}=g(2\pi).
\]
So we have a contradiction between \eqref{eqn:ex2c4} and \eqref{eqn:ex2c5}.
Hence, the diffeomorphism \eqref{eqn:ex2c1} is not $C^1$ integrable
 although it is near-identity when $0<\epsilon\ll 1$.
\end{ex}

\begin{rmk}
Example~$\ref{ex:2c}$ also shows that the diffeomorphism $\eqref{eqn:ex2c1}$ cannot be represented
 by the flow of a $C^\infty$ vector field on $\Sset^1$, i.e.,
 it cannot belong to any one-parameter subgroup of 
 the group $\mathrm{Diff}(\Sset^1)$ consisting of all $C^\infty$ diffeomorphisms of $\Sset^1$.
Indeed, it is $(1,0)$-integrable in the sense of Definition~$\ref{dfn:2a}$ otherwise.
This fact was proven in a different manner in Warning~$1.6$ of {\rm\cite{M84}}.
\end{rmk}

Thus, one-dimensional diffeomorphisms are not always integrable
 in {\color{black}the sense of} Definition~$\ref{dfn:2a}$.
{\color{black}This is in contrast with the fact that}
  one-dimensional vector fields are always $(1,0)$-integrable
 in {\color{black}the sense of} Definition~$\ref{dfn:1b}$.
 
{\color{black}We also give a genuine higher-dimensional example from the literature.}

\begin{ex}
Consider the $n$-dimensional Lyness map {\rm\cite{KLR93,CGM07,CGM08a,CGM08b,BR09}},
\begin{equation}
f(x)=\left(x_2,\ldots,x_n,\frac{\sum_{j=2}^{n}x_j+a}{x_1}\right),\quad
x=(x_1,\ldots,x_n)\in(0,\infty)^n,
\label{eqn:lm}
\end{equation}
which was originally introduced by Lyness {\rm\cite{L42}},
 {\color{black}subsequently discussed for $n=2$} in, e.g., {\rm\cite{BC98,BR04}},  
 and generalized to $n\ge 3$ {\rm\cite{KLR93}}.
It is a $C^\infty$ diffeomorphism from $(0,\infty)^n$ to $(0,\infty)^n$ since
\[
f^{-1}(x)=\left(\frac{\sum_{j=1}^{n-1}x_j+a}{x_n},x_1,\ldots,x_{n-1}\right).
\]
%
The diffeomorphism \eqref{eqn:lm} has analytic first integrals
\[
F_1(x)=\frac{\left(\sum_{j=1}^nx_j+a\right)\prod_{j=1}^n(x_j+1)}{x_1\cdots x_n}
\]
for $n\ge 2$, and
\[
F_2(x)=\frac{\left(\sum_{j=1}^nx_j+x_1x_n+a\right)\prod_{j=1}^{n-1}(x_j+x_{j+1}+1)}{x_1\cdots x_n}
\]
for $n\ge 3$. 
These first integrals are functionally independent a.e. if they exist.
It also has an analytic commutative vector field $v_1(x)=(v_{11}(x),\ldots,v_{1n}(x))$, where
\begin{align*}
&
v_{11}(x)=\frac{(x_1+1)\left(\sum_{j=1}^{n-1}x_j-x_2x_n\right)\prod_{j=2}^{n-1}(x_j+x_{j+1}+1)}{x_1\cdots x_n},\\
&
v_{1n}(x)
=\frac{(x_n+1)\left(\sum_{j=2}^{n-1}x_j-x_1x_{n-1}\right)\prod_{j=1}^{n-2}(x_j+x_{j+1}+1)}{x_1\cdots x_n}
\end{align*}
and
\[
v_{1l}(x)
=\frac{(x_l+1)\left(\sum_{j=1}^{n-1}x_j+x_1x_n\right)(x_{l-1}-x_{l+1})\prod_{j=1,i\neq l-1,l}^{n-1}(x_j+x_{j+1}+1)}
 {x_1\cdots x_n}
\] 
for $2\le l\le n-1$, as easily shown by Proposition~$\ref{prop:2d}$.
In addition, we compute
\[
\D F_j(x)\cdot v_1(x)\equiv 0,\quad
j=1,2,
\]
if the first integrals exist, for $n\le 5$.
See {\rm\cite{CGM08a}} for the details.
Thus, the diffeomorphism \eqref{eqn:lm} is analytically $(1,1)$-integrable for $n=2$,
 and $(1,2)$-integrable for $n=3$.
See also {\rm\cite{CGM07,CGM08b}}.
\end{ex}

Finally, we consider the linear diffeomorphism
\begin{equation}
f(x)=Ax
\label{eqn:ldif}
\end{equation}
in $\Cset^n$, where $A\in\mathrm{GL}(n,\Cset)$.

\begin{prop}
\label{prop:2c}
The linear diffeomorphism \eqref{eqn:ldif} is {\color{black}complex-}analytically $(n,0)$-integrable.
\end{prop}

\begin{proof}
We only have to show the statement
 when the coefficient matrix $A$ is given by the Jordan normal form
\[
A=\begin{pmatrix}
\lambda &&&  O\\
1 &\ddots &\\
&\ddots& \lambda&\\
O && 1 & \lambda
\end{pmatrix},
\]
where $\lambda\in\Cset$.

Let
\begin{equation}
v_1(x)=Ax,\quad
v_{j+1}(x)=\sum_{k=1}^{n-j}x_ke_{k+j},\quad
j\in[n-1],
\label{eqn:p3a}
\end{equation}
where $\{e_\ell\}_{\ell=1}^n$ denotes the standard basis in $\Cset^n$.
We easily see that $v_j(x)$, $j\in[n]$, are linearly independent a.e.
 and commute with each other.
Let $\varphi_j^t(x)$ be the flows of $v_j(x)$ for $j\in[n]$.
We easily see that
\[
Av_j(x)-v_j(Ax)\equiv 0,\quad
j\in[n].
\]
Integrating the above relation with respect to $t$, we have
\[
A\varphi_j^t(x)-\varphi_j^t(Ax)\equiv 0,\quad
j\in[n],
\]
since $A\varphi_j^t(x)-\varphi_j^t(Ax)=0$ at $t=0$.
This completes the proof.
\end{proof}

\begin{rmk}\
\begin{enumerate}
\setlength{\leftskip}{-1.6em}
\item[\rm(i)]
In Proposition~$\ref{prop:2c}$, we make no special assumption on the matrix $A$
 unlike {\rm\cite{Z13,JS21}}.
{\color{black}
\item[\rm(ii)]
If the matrix $A$ is real and its eigenvalues are all real,
 then we can choose $\Rset^n$ as its phase space,
 so that the linear diffeomorphism \eqref{eqn:ldif} is real-analytically $(n,0)$-integrable.}
\end{enumerate}
\end{rmk}

Thus, linear diffeomorphisms are always integrable {\color{black}in the sense of} Definition~\ref{dfn:2a},
{\color{black}as is the case for linear vector fields in the sense of Definition~\ref{dfn:1b}.}
A surprising result which shows that
 many nonlinear diffeomorphisms are integrable in certain regions of their phase spaces
 will be reported in \cite{Y25b}.


\section{Relation of Integrability with Regular and Irregular Dynamics}

We now state our extension of the Liouville-Arnold theorem to diffeomorphisms.

\begin{thm}
\label{thm:3a}
Let $f:M^n\to M^n$ be a $C^r$ $(m,n-m)$-integrable diffeomorphism with $m\ge 1$,
 and let $F_j$, $j\in[n-m]$, denote its $n-m$ first integrals.
Suppose that the level set $F^{-1}(c)$ with $c\in\Rset^{n-m}$ is connected,
 where $F(x)=(F_1(x),\ldots,F_{n-m}(x))$,
 and that the $m$ commutative vector fields are linearly independent on it.
Then $F^{-1}(c)$ is invariant under $f$ and {\color{black}$C^r$} diffeomorphic to a toroidal cylinder $\Tset^\ell\times\Rset^{m-\ell}$
 for some $\ell\in[m]\cup\{0\}$, and $f$ is $C^r$ conjugate to
\begin{equation}
(I,\theta)\mapsto(I,\theta+\Theta(I)),\quad
I\in U_I,\quad \theta\in\Tset^\ell\times\Rset^{m-\ell},
\label{eqn:thm3a}
\end{equation}
in a neighborhood of $F^{-1}(c)$, where $U_I$ is a neighborhood of $I=0$ in $\Rset^{n-m}$
 and $\Theta: U_I\to\Tset^\ell\times\Rset^{m-\ell}$ is $C^r$.
\end{thm}

\begin{proof}
Using discussions in Section~49 of \cite{A89},
 we  {\color{black}first} show, as in the case of vector fields, 
 that the level set $F^{-1}(c)$ is diffeomorphic to $\Tset^\ell\times\Rset^{m-\ell}$
 for some $\ell\in[m]\cup\{0\}$.
 
Suppose that the hypotheses of the theorem hold.
Let $X_j$, $j\in[m]$, be the {\color{black}linearly independent,} commutative vector fields {\color{black}for $f$}
 and let $\varphi_j^t$, $j\in[m]$, be their flows.
{\color{black}
Since $F_k$, $k\in[n-m]$, are first integrals of $X_j$,
 the vector field $X_j$ is tangent to $F^{-1}(c)$ for $j\in[m]$.
Hence, $\varphi_j^t$, $j\in[m]$, restrict to $F^{-1}(c)$
 and commute with each other.}
Let $\bar{t}=(t_1,\ldots, t_m)\in\Rset^m$ and define
\[
\varphi^{\bar{t}}=\varphi_1^{t_1}\circ\cdots\circ\varphi_m^{t_m}.
\]
{\color{black}
For $\bar{x}\in F^{-1}(c)$ fixed,
 let $\varphi_{\bar{x}}(\bar{t})=\varphi^{\bar{t}}(\bar{x})$.
Then $\varphi_{\bar{x}}$ is $C^{r+1}$.
Hence, $\varphi_{\bar{x}}$ gives a $C^{r+1}$ local chart near $\bar{x}$ in $F^{-1}(c)$.
Since $F^{-1}(c)$ is connected,
 any point $x\in F^{-1}(c)$ can be joined to $\bar{x}$ by a curve in $F^{-1}(c)$.
This curve is covered by finitely many such neighborhoods of points $\bar{x}_j$, $j\in[k]$, lying on it
 for some $k\in\Nset$.
Passing successively through these neighborhoods along the curve,
 we obtain $\bar{t}$ such that $x=\varphi^{\bar{t}}(\bar{x})$,
 where $\bar{t}$ is given by the sum of the shifts corresponding to the pieces of the curve.
Hence, $\varphi_{\bar{x}}$ is also onto $F^{-1}(c)$.

Let $\Gamma_{\bar{x}}=\{\bar{t}\in\Rset^m\mid \varphi^{\bar{t}}(\bar{x})=\bar{x}\}$.
Since $\varphi_{\bar{x}}$ is a local diffeomorphism at $\bar{t}=0$,
 there is a neighborhood $U$ of $\bar{t}=0$ in $\Rset^m$ 
 such that $U\cap\Gamma_{\bar{x}}=\{0\}$.
Thus, $\Gamma_{\bar{x}}$ is a discrete subgroup of $\Rset^m$.
Hence, we find an integer $\ell\in[m]\cup\{0\}$ and linearly independent vectors
$\bar e_j\in\Rset^m$, $j\in[\ell]$, such that
\begin{equation}
\Gamma_{\bar{x}}=\Zset\bar e_1+\cdots+\Zset\bar e_\ell .
\label{eqn:thm3ap}
\end{equation}
In particular, the isotropy subgroup $\Gamma_{\bar{x}}$ is independent of the choice of
$\bar{x}$. Indeed, for any $x\in F^{-1}(c)$ there is
$\bar{t}\in\Rset^m$ such that $x=\varphi^{\bar{t}}(\bar{x})$.
If $\bar{s}\in\Gamma_{\bar{x}}$, then
\[
\varphi^{\bar{s}}(x)
=\varphi^{\bar{s}}(\varphi^{\bar{t}}(\bar{x}))
=\varphi^{\bar{t}}(\varphi^{\bar{s}}(\bar{x}))
=\varphi^{\bar{t}}(\bar{x})=x.
\]
Similarly, if $\bar{s}\in\Gamma_{x}$, then $\varphi^{\bar{s}}(\bar{x})=\bar{x}$.
Thus, we have $\Gamma_x=\Gamma_{\bar{x}}$.}

Choose $\bar{e}_{j+\ell}$, $j\in[m-\ell]$,
 such that $\{\bar{e}_j\}_{j=1}^m$ becomes a basis {\color{black}of} $\Rset^m$.
Let $\{e_j\}_{j=1}^m$ be the standard basis {\color{black}of} $\Rset^m$
 and let $B$ be an $m\times m$ invertible matrix such that
\[
Be_j=\begin{cases}
\bar{e}_j/2\pi  & \mbox{for $j\le\ell$};\\
\bar{e}_j & \mbox{for $j>\ell$}.
\end{cases}
\]
Let $\pi:\Rset^m\to \Tset^\ell\times\Rset^{m-\ell}$ be the natural map
 such that $\pi(s)=\theta$ with
\[
\theta_j=\begin{cases}
s_j \mod\ 2\pi &  \mbox{for $j\le\ell$};\\
s_j & \mbox{for $j>\ell$},
\end{cases}
\]
where $s=(s_1,\ldots,s_m)$ and $\theta=(\theta_1,\ldots,\theta_m)$.
{\color{black}
From \eqref{eqn:thm3ap} we see that
 if $\pi(s)=\pi(s')$, then $B(s-s')\in\Gamma_{\bar{x}}$
  and consequently $\varphi^{Bs}(\bar{x})=\varphi^{Bs'}(\bar{x})$.
Recall that $\varphi_{\bar{x}}$ is onto $F^{-1}(c)$ and locally a diffeomorphism,
 and that it has $\Gamma_{\bar{x}}$ as its isotropy subgroup.
Thus, we can define a diffeomorphism $h:\Tset^\ell\times\Rset^{m-\ell}\to F^{-1}(c)$}
 such that the following diagram commutes:
\[
\begin{CD}
s\in\Rset^m @>B>> \bar{t}\in\Rset^m\\
@V{\pi}VV @VV{\color{black}\varphi_{\bar{x}}}V\\
\theta\in\Tset^\ell\times\Rset^{m-\ell}@>{h}>>F^{-1}(c)\subset M^n
\end{CD}
\]
Thus, $F^{-1}(c)$ is diffeomorphic to $\Tset^\ell\times\Rset^{m-\ell}$.

It remains to show that $f$ is {\color{black}conjugate} to \eqref{eqn:thm3a} near $F^{-1}(c)$.
First, we see that there exists $\bar{t}_0\in\Rset^m$
 such that $f(\bar{x})=\varphi^{\bar{t}_0}(\bar{x})$,
 since {\color{black}$\varphi_{\bar{x}}$ is onto $F^{-1}(c)$.}
Here $\bar{t}_0$ is independent of the choice of $\bar{x}$.
Indeed, when $x=\varphi^{\bar{t}}(\bar{x})\in F^{-1}(c)$, we have
\begin{equation}
f(x)=f(\varphi^{\bar{t}}(\bar{x}))
 =\varphi^{\bar{t}}(f(\bar{x}))=\varphi^{\bar{t}}(\varphi^{\bar{t}_0}(\bar{x}))
 =\varphi^{\bar{t}_0}(\varphi^{\bar{t}}(\bar{x}))=\varphi^{\bar{t}_0}(x).
\label{eqn:pthm3a}
\end{equation}
Let $\theta_0=\pi(B^{-1}\bar{t}_0)$.
Then we have the following commutative diagram:
\[
\begin{CD}
\theta\in\Tset^\ell\times\Rset^{m-\ell} @>{h}>> F^{-1}(c)\\
@V{}VV @VV{f}V\\
\theta+\theta_0\in\Tset^\ell\times\Rset^{m-\ell}@>{h}>>F^{-1}(c)
\end{CD}
\]
Hence, we obtain
\begin{equation*}
h(\theta+\theta_0)=f(h(\theta)).
\end{equation*}
{\color{black}
The same construction applies to the regular level sets
 $F^{-1}(I)$ for $I$ near $c\in\Rset^{n-m}$.
Since $F$ and the flows $\varphi_j^t$, $j\in[m]$, are $C^r$,
 the above coordinates $\theta$ and the corresponding translation vector $\theta_0$
 can be chosen $C^r$ smoothly with respect to $I$.
Letting $\theta_0=\Theta(I)$,}
 we obtain \eqref{eqn:thm3a} near $F^{-1}(c)$.
\end{proof}

\begin{rmk}\
\begin{enumerate}
\setlength{\leftskip}{-1.6em}
\item[\rm(i)]
In Theorem~$\ref{thm:3a}$,
 if the level set $F^{-1}(c)$ is compact,
 then $\ell=m$ and $F^{-1}(c)$ is diffeomorphic to $\Tset^m$,
 as in the classical Liouville-Arnold theorem $($see, e.g., Section~{\rm 49A} of {\rm\cite{A89})}.
\item[\rm(ii)]
If $m=0$ in Theorem~$\ref{thm:3a}$, then $F^{-1}(c)$ reduces to a single point.
\end{enumerate} 
\end{rmk}

{\color{black}
The translation form \eqref{eqn:pthm3a} obtained in the proof of Theorem~\ref{thm:3a}
 also leads to an inclusion of the diffeomorphism
 in the flow of an integrable vector field, as follows.}

\begin{cor}
Under the hypotheses of Theorem~$\ref{thm:3a}$,
 the diffeomorphism $f:M^n\to M^n$
 is the time-one map of an $(m,n-m)$-integrable vector field
  in the sense of Bogoyavlenskij $($see Definition~$\ref{dfn:1b})$.
\end{cor}

\begin{proof}
In the proof of Theorem~\ref{thm:3a},
 let $\bar{t}_0=(t_{10},\ldots,t_{m0})$ in \eqref{eqn:pthm3a},
 and let $\bar{\varphi}^t$ denote the flow of the vector field
\begin{equation}
\bar{X}(x)=\sum_{j=1}^m{\color{black}t_{j0}}X_j(x).
\label{eqn:cor}
\end{equation}
{\color{black}Then we have
\[
\bar{\varphi}^t(x)
=\varphi_1^{t_{10}t}\circ\cdots\circ\varphi_m^{t_{m0}t}(x).
\]
Indeed, by the commutativity of $\varphi_j^t$, $j\in[m]$, 
\begin{align*}
\frac{\d}{\d t}\varphi_1^{t_1}\circ\cdots\circ\varphi_m^{t_m}(x)
 =&\sum_{j=1}^m t_{j0}\frac{\partial\varphi_j^{t_j}}{\partial t_j}
 (\varphi_1^{t_1}\circ\cdots\circ\varphi_{j-1}^{t_{j-1}}\circ
 \varphi_{j+1}^{t_{j+1}}\circ\cdots\circ\varphi_{m}^{t_m}(x)),
\end{align*}
which yields the vector field in \eqref{eqn:cor} at $t=0$,
 where $t_j=t_{j0}t$, $j\in[m]$.
Hence,} $f$ is the time-one map of $\bar{X}$ by \eqref{eqn:pthm3a}.
{\color{black}Moreover,} $\bar{X}$ is $(m,n-m)$-integrable
 since $X_j$, $j\in[m]$, are so.
Thus, we obtain the desired result.
\end{proof}

We naturally expect that the presence of chaotic dynamics obstructs integrability.
In particular, if a diffeomorphism exhibits strong forms of dynamical complexity,
 then it should fail to be integrable in the sense introduced above.
We now make this intuition precise.

\begin{prop}
\label{prop:2b}
Let $f:M^n\to M^n$ be a $C^1$ diffeomorphism.
Suppose that the following conditions hold$\,:$
\begin{enumerate}
\setlength{\leftskip}{-1.6em}
\item[\rm(i)]
$f$ is topologically transitive,
 i.e., there exists a dense orbit $\O(x^*)=\{f^k(x)\mid k\in\Zset\}$ with some $x^*\in M^n;$
\item[\rm(ii)]
Every element of $P(f)$ is hyperbolic and $P(f)$ is dense,
 where $P(f)$ denotes the set of periodic points of $f$.
\end{enumerate}
Then $f$ is not $C^1$ integrable.
\end{prop}

\begin{proof}
Let $F$ be a {\color{black}$C^1$} first integral
 and let $x^*\in M^n$ such that $\overline{\O(x^*)}=M^n$.
Since $F(\O(x^*))=\{F(x^*)\}$, we have
\[
F(M^n)=F\bigl(\overline{\O(x^*)}\bigr)=\overline{F(\O(x^*))}=\{F(x^\ast)\}.
\]
Hence, $F(x)$ is a constant.
Thus, $f$ has no nonconstant {\color{black}$C^1$} first integral.

Let $\varphi^t$ be a $C^1$ flow commuting with $f$,
 and let $x^k$ be a $k$-periodic point for some $k\in\Nset$.
Since
\[
f^k\circ \varphi^t(x^k)=\varphi^t\circ f^k(x^k)=\varphi^t(x^k),
\]
$\varphi^t(x^k)$ is also a $k$-periodic point of $f$ for any $t\in\Rset$.
However, by hyperbolicity of $x^k$, $\varphi^t(x^k)=x^k$ for any $t\in\Rset$
 since $x^k$ is not isolated otherwise.
Hence, we have $\varphi^t|_{P(f)}=\id$, where `$\id$' represents the identity map.
Since $\varphi^t$ is $C^1$ for any $t\in\Rset$ and $P(f)$ is dense,
 we obtain $\varphi^t=\id$ in $M^n$.
Thus, $f$ has no nonzero $C^1$ commutative vector field.
This completes the proof.
\end{proof}

\begin{ex}
\label{ex:2d}
Consider Arnold's cat map {\rm\cite{AA68}:}
\[
f(x)=Ax,\quad
A=\begin{pmatrix}
2 & 1\\
1 & 1
\end{pmatrix},\quad
x\in\tilde{\Tset}^2:=\Rset^2/\Zset^2.
\]
Since it is topologically transitive and $P(f)$ is dense as shown easily
 $($see, e.g., Section~{\rm III.5} of {\rm\cite{Z10})},
 $f:\tilde{\Tset}^2\to\tilde{\Tset}^2$ is not $C^1$ integrable by Proposition~$\ref{prop:2b}$.
Note that if its phase space is $\Rset^2$,
 then by Proposition~$\ref{prop:2c}$, 
{\color{black}the map} $f(x)=Ax$ is analytically $(2,0)$-integrable.
\end{ex}


Let $\Sigma=\{\{a_j\}_{j=-\infty}^\infty\mid a_j\in[2],j\in\Zset\}$
 be the set of bi-infinite sequences of two symbols with the metric
\[
d(a,b)=\sum_{j=-\infty}^\infty 2^{-|j|}|a_j-b_j|
\]
for $a,b\in\Sigma$.
Let $\sigma:\Sigma\to\Sigma$ be the shift map such that
\[
\sigma(a)_j=a_{j+1},\quad
j\in\Zset,
\]
for any $a\in\Sigma$.

\begin{prop}
\label{prop:3a}
Let $f:M^2\to M^2$ be an {\color{black}analytic} diffeomorphism
 having a hyperbolic  invariant set $\Lambda$
 on which it is topologically conjugate to $\sigma:\Sigma\to\Sigma$.
Then $f$ is not analytically integrable. 
\end{prop}

\begin{proof}
Suppose that the hypothesis holds but $f$ is analytically $(m,2-m)$-integrable
 for some $m\in[2]\cup\{0\}$.
We see that $f$ has a dense orbit in $\Lambda$
 since so does $\sigma$ in $\Sigma$.
Let $\O(x^*)$ be such a dense orbit for $x^*\in\Lambda$.
Moreover, the set of periodic points of $f$ on $\Lambda$,
 $P\bigl(f|_\Lambda\bigr)$, is dense in $\Lambda$,
 and contains periodic points of any period,
 since so does $P(\sigma)$ in $\Sigma$.

\begin{figure}[t]
\includegraphics[scale=0.6]{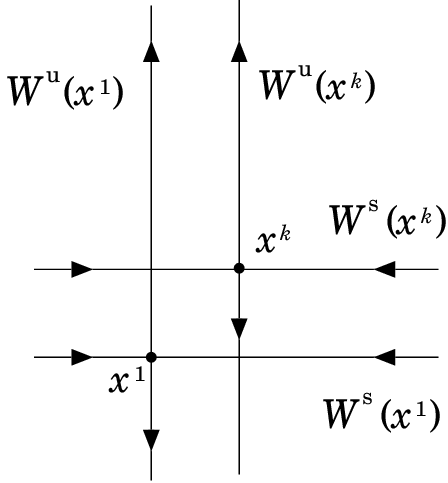}
\caption{Proof of Proposition~\ref{prop:3a}.
$W^\s(x)$ and $W^\u(x)$ denote pieces of the stable and unstable manifolds
 of $x=x^1,x^k\in P(f)$.
Here $x^k$ represents a $k$-periodic point.
\label{fig:3a}}
\end{figure}

Let $F$ be an analytic first integral.
Since
\[
F(\Lambda)=F\bigl(\overline{\O(x^*)}\bigr)=\overline{F(\O(x^*))}=\{F(x^\ast)\},
\]
we have $F(x)=F(x^*)=:c$ on $\Lambda$.
Let $x^1$ be a fixed point, which is hyperbolic by assumption.
We find a hyperbolic periodic point in any neighborhood of $x^1$.
Including $x^1$, these periodic points have analytic stable and unstable manifolds
(see, e.g., Theorem~7.1 in Chapter~1 of \cite{IY07}),
 on which infinitely many points are contained by $\Lambda$,
 and they are accumulation points of the infinite sets.
See Fig.~\ref{fig:3a}.
Since $F(x)$ is analytic, we have $F(x)=c$ on the stable and unstable manifolds
 as well as at these periodic points, by the identity theorem of one variable,
 so that it also holds in a neighborhood of $x^1$.
Indeed, we have $F(x)=c$ on any sequence accumulating to $x^1$.
Hence, it follows from the identity theorem of several variables (see, e.g., Theorem~1A.6 of \cite{RH65})
 that $F(x)=c$ on $M^2$.
This means that $f$ has no analytic first integral.

Let $\varphi^t$ be the flow of a commutative vector field,
 and let $x^k\in P(f_\Lambda)$ with a period $k$.
As in the proof of Proposition~\ref{prop:2b},
 we can show that $\varphi^t(x^k)=x^k$ and consequently $\varphi^t=\id$ on $\Lambda$
 for any $t\in\Rset$.
Moreover, using the above argument,
 we obtain $\varphi^t=\id$ on $M^2$.
This completes the proof.
\end{proof}

\begin{rmk}
It is not difficult to extend Proposition~$\ref{prop:3a}$
 to some cases of higher-dimensional diffeomorphisms, for instance,
 in which there exists a two-dimensional, normally hyperbolic invariant manifold {\rm\cite{HPS77}}
 containing the invariant set $\Lambda$.
\end{rmk}

\begin{ex}
\label{ex:3a}
Let $S=[0,1]^2$ and let $f:S\to\Rset^2$ be the Smale horseshoe
 $($see, e.g., Section {\rm I.5} of {\rm\cite{S67}} or Section~$5.1$ of {\rm\cite{GH83})}.
Since it satisfies the hypothesis of Proposition~$\ref{prop:3a}$
 $($see, e.g., Theorem~$5.1.1$ of {\rm\cite{GH83})},
 $f$ is not analytically integrable.
\end{ex}

It is well known that a two-dimensional diffeomorphism contains a horseshoe 
 if it has a hyperbolic periodic orbits
 whose stable and unstable manifolds intersect transversely.
See, e.g., Theorem~5.3.5 of \cite{GH83} or Theorem~III.17 of \cite{Z10}.
Such a diffeomorphism is not analytically integrable, as in Example~\ref{ex:3a}
(cf. Theorem~3.8 of \cite{M73}).


\section{Symplectic Diffeomorphisms}

Finally, we consider symplectic diffeomorphisms
 and discuss some relations {\color{black}between their standard integrability
 and our definition of integrable diffeomorphisms in Definition~\ref{dfn:2a}.}

We first briefly review previous results on symplectic diffeomorphisms and their integrability
 \cite{BRSG91,V91,HJN16}.
Here, for simplicity,
 we restrict ourselves to symplectic diffeomorphisms of the form $(Q(q,p),P(q,p))$,
 where $Q,P:\Rset^n\times\Rset^n\to\Rset^n$ are $C^r$ and $q,p\in\Rset^n$,
 as in \cite{BRSG91}.
In addition, the domain $\Rset^n\times\Rset^n$ is equipped with the canonical symplectic form
\[
\Omega=\sum_{j=1}^n\d p_j\wedge\d q_j,
\]
where $q_j,p_j$ are the $j$-th components of $q,p$ for $j\in[n]$.
By Darboux's theorem (see, e.g., Section~43 of \cite{A89}),
 {\color{black}the local discussions below also apply to general symplectic diffeomorphisms,
as  in the case of} Hamiltonian vector fields in Section~2.
Recall that the diffeomorphism $(Q(q,p),P(q,p))$ is called \emph{symplectic}
 if it preserves the symplectic form $\Omega$.

Let $\{(q(i),p(i))\}_{i\in\Zset}$ denote an orbit of the discrete dynamical system
 defined by the symplectic map $(Q(q,p),P(q,p))$, i.e.,
\begin{equation}
q(i+1)=Q(q(i),p(i)),\quad
p(i+1)=P(q(i),p(i)),\quad
i\in\Zset.
\label{eqn:dsys}
\end{equation}
Letting $(q',p')=(Q(q,p),P(q,p))$, we have $\d p\wedge\d q=\d p'\wedge\d q'$, so that
\[
\d(p\d q+q'\d p')=\d(p\d q-p'\d q'+\d(q'p'))=0.
\]
By Poincar\'{e}'s lemma (e.g., Section~36 of \cite{A89}),
 there exists a function $S(q,p')$ {\color{black}locally} such that
\[
p\d q+q'\d p'=\d S.
\]
Hence, we have the relation
\[
q(i+1)=\D_{p'}S(q(i),p(i+1)),\quad
p(i)=\D_qS(q(i),p(i+1)),\quad
i\in\Zset,
\]
for the orbit $\{(q(i),p(i))\}_{i\in\Zset}$.
The function $S(q,p')$ is called a \emph{generating function}.
We rewrite $H(q,p')=qp'+S(q,p')$ to obtain the following.

\begin{prop}
\label{prop:4a}
There exists a $C^{r+1}$ scalar-valued function $H(q,p)$ {\color{black}locally}
 such that the system \eqref{eqn:dsys} is written as
\begin{equation}
\begin{split}
q(i+1)=&q(i)+\D_pH(q(i),p(i+1)),\\
p(i+1)=&p(i)-\D_qH(q(i),p(i+1)).
\end{split}
\label{eqn:dsys1}
\end{equation}
\end{prop}

The system \eqref{eqn:dsys1} is very similar
 to the canonical Hamiltonian vector field \eqref{eqn:Hsys1}.
We now state the standard definition of integrability of symplectic diffeomorphisms
 for \eqref{eqn:dsys}.

\begin{dfn}[Integrability of symplectic diffeomorphisms]
\label{dfn:4a}
We say that the symplectic diffeomorphism $(Q(q,p),P(q,p))$ is \emph{integrable}
 if {\color{black}there exist} $n$ {\color{black}$C^r$} scalar-valued functions {\color{black}$F_j(q,p)$, $j\in[n]$,}
 such that the derivatives {\color{black}$\D F_j(q,p)$, $j\in[n]$,} are linearly independent a.e.
 and they are in involution, i.e.,
\[
\{F_j,F_k\}(q,p):=\D_qF_j(q,p)^\T\D_pF_k(q,p)-\D_pF_j(q,p)^\T\D_qF_k(q,p)\equiv 0\quad\mbox{for $j,k\in[n]$}.
\]
\end{dfn}

Suppose that the symplectic diffeomorphism $(Q(q,p),P(q,p))$ is integrable
 and the first integrals $F_j(q,p)$, $j\in[n]$, are functionally independent.
Let $F(q,p)=(F_1(q,p),\ldots,F_n(q,p))$ and let $\hat{p}=F(q,p)$. 
{\color{black}Redefining the variables $q,p$ if necessary,
 we assume that $\D_p F(q,p)$ is nonsingular in a neighborhood,
 since $\D F_j(q,p)$, $j\in[n]$,} are linearly independent {\color{black}a.e.}
Hence, $F(q,p)$ is {\color{black}locally} invertible for $q$ fixed,
 {\color{black}and} there exists a {\color{black}$C^{r}$} function $\phi(\hat{p},q)$ such that $\hat{p}=F(q,\phi(\hat{p},q))$.
Define a {\color{black}$C^{r}$} generating function as
\[
S(\hat{p},q)=\sum_{j=1}^n\int\phi_j(\hat{p},q)\d q_j,
\]
where $\phi_j(\hat{p},q)$ is the $j$-th {\color{black}component} of $\phi(\hat{p},q)$ for $j\in[n]$,
 and let $\hat{q}=\D_{\hat{p}}S(\hat{p},q)$.
Hence, the diffeomorphism $(Q(q,p),P(q,p))$ is {\color{black}locally} transformed
 into a symplectic diffeomorphism under the {\color{black}$C^{r-1}$} change of coordinates from $(q,p)$ to $(\hat{q},\hat{p})$.
{\color{black}Thus,}
 we can  {\color{black}locally} write the transformed diffeomorphism
 {\color{black}in the form of} \eqref{eqn:dsys} as
\begin{equation}
\hat{q}(i+1)=\hat{q}(i)+\D H(\hat{p}(i+1)),\quad
\hat{p}(i+1)=\hat{p}(i),
\label{eqn:idsys}
\end{equation}
since $\hat{p}(i)$ is {\color{black}preserved}
 and consequently the generating function $H$
 does not depend {\color{black}on} $\hat{q}(i)$.
Thus, the symplectic map \eqref{eqn:idsys} is  {\color{black}locally} solvable,
 and its dynamics are very simple like \eqref{eqn:thm3a}.
We also have the following.

\begin{prop}
\label{prop:4b}
{\color{black}
Suppose that the $2n$-dimensional symplectic diffeomorphism \eqref{eqn:dsys}
 is $C^r$ integrable  in the sense of Definition~$\ref{dfn:4a}$ with $r\ge 2$.
Then it is locally $C^{r-1}$ $(n,n)$-integrable in the sense of Definition~$\ref{dfn:2a}$
 near any point $(q,p)$ at which the derivatives of the first integrals are linearly independent.}
\end{prop}

\begin{proof}
Suppose that the diffeomorphism \eqref{eqn:dsys} is {\color{black}$C^r$} integrable in the sense of Definition~\ref{dfn:4a}.
Then it is {\color{black}locally} transformed into \eqref{eqn:idsys} by a {\color{black}$C^{r-1}$} transformation,
 as described above.

It is clear that each component of $\hat{p}$ is an analytic first integral of \eqref{eqn:idsys}.
Let $\{e_j\}_{j=1}^n$ be the standard basis in $\Rset^n$.
Since
\[
\begin{pmatrix}
\id_n & \D^2 H(\hat{p}(i))\\
0 & \id_n
\end{pmatrix}
\begin{pmatrix}
e_j\\
0
\end{pmatrix}
=\begin{pmatrix}
e_j\\
0
\end{pmatrix},\quad
j\in[n],
\]
we see by Proposition~\ref{prop:2d}
 that $v_j({\color{black}\hat{q},\hat{p}})=(e_j,0)$, $j\in[n]$, are commutative vector fields of \eqref{eqn:idsys}.
Since each component of $\hat{p}$ is obviously an analytic first integral
 of $v_j({\color{black}\hat{q},\hat{p}})$, $j\in[n]$,
 the diffeomorphism \eqref{eqn:idsys} is {\color{black}$C^{r-1}$} $(n,n)$-integrable.
Since the {\color{black}local} transformation from \eqref{eqn:dsys} to \eqref{eqn:idsys} is {\color{black}$C^{r-1}$}
 and the diffeomorphism \eqref{eqn:dsys} has $n$ {\color{black}$C^r$} first integrals by assumption,
 we obtain the desired result.
\end{proof}
 
We return to a general $C^r$ diffeomorphism $f:M^n\to M^n$,
 where $M^n$ is an $n$-dimensional $C^\infty$ or analytic manifold, as in Sections~3 and 4.
We first modify the \emph{cotangent lift trick},
 which was explained shortly for vector fields in Section~2, for diffeomorphisms.

In a local coordinate system,
 let $(x,p)\in T^*M^n$ and let $H(x,p')=p'\cdot(-x+f(x))$.
We have the $2n$-dimensional symplectic diffeomorphism
\begin{equation}
x(i+1)=f(x(i)),\quad
p(i)=\D f(x(i))^\T p(i+1)
\label{eqn:dsys2}
\end{equation}
in the form of \eqref{eqn:dsys1}.
Note that the first component of \eqref{eqn:dsys2} is the same as the original one
(cf. Eq.~\ref{eqn:Hsys2}).
We have a counterpart of Proposition~\ref{prop:1a} for diffeomorphisms as follows.

\begin{prop}
\label{prop:4c}
If the diffeomorphism $f:M^n\to M^n$ is $C^r$ integrable
 in the sense of Definition~$\ref{dfn:2a}$,
 then the symplectic diffeomorphism \eqref{eqn:dsys2} is  $C^r$ integrable
 in the sense of Definition~$\ref{dfn:4a}$.
\end{prop}

\begin{proof}
Suppose that $f$ is $(m,n-m)$-integrable
 and let $v_j(x)$, $j\in[m]$, and $F_j(x)$, $j\in[n-m]$,
 be its commutative vector fields and first integrals, respectively.
Obviously, $F_j(x)$, $j\in[n-m]$, are first integrals of \eqref{eqn:dsys2}.
Let $G_j(x,p)=p\cdot v_j(x)$, $j\in[m]$.
Using \eqref{eqn:dsys2} and Proposition~\ref{prop:2d}, we compute
\begin{align*}
G_j(x(i),p(i))=&p(i)\cdot v_j(x(i))=\D f(x(i))^\T p(i+1)\cdot v_j(x(i))\\
=&p(i+1)\cdot\D f(x(i))v_j(x(i))=p(i+1)\cdot v_j(f(x(i)))\\
=&p(i+1)\cdot v_j(x(i+1))=G_j(x(i+1),p(i+1)),\quad
j\in[m].
\end{align*}
Hence, $G_j(x,p)$, $j\in[m]$, are first integrals of \eqref{eqn:dsys2}.

It remains to show that these first integrals are in involution.
Obviously,
\[
\{F_j(x),F_k(x)\}(x,p)\equiv 0,\quad
j,k\in[n-m].
\]
We compute
\begin{align*}
\{F_j(x),G_k(x,p)\}(x,p)=&\D_xF_j(x)^\T\D_pG_k(x,p)\\
=&\D_x F_j(x)^\T v_k(x)\equiv 0,\quad
j\in[n-m],\ k\in[m],
\end{align*}
and
\begin{align*}
\{G_j(x,p),G_k(x,p)\}(x,p)=&\D_xG_j(x,p)^\T\D_pG_k(x,p)-\D_pG_j(x,p)^\T\D_xG_k(x,p)\\
=&p^\T(\D v_j(x)v_k(x)-\D v_k(x)v_j(x))\equiv 0,\quad
j,k\in[m],
\end{align*}
{\color{black}by conditions~(i) and (ii) of Definition~\ref{dfn:2a}.}
Thus, we obtain the desired result.
\end{proof}

\section*{Acknowledgments}
The author thanks Shoya Motonaga for helpful comments and discussions,
 especially on Definition~\ref{dfn:2a}
 and the proofs of Theorem~\ref{thm:3a} and Proposition~\ref{prop:2b}.
In particular, the proof of Proposition~\ref{prop:2b} is due to him.
{\color{black}The author is also grateful to the anonymous referees for their careful reading
 and constructive comments, which were valuable during the revision.}
This work was partially supported by the JSPS KAKENHI Grant Number JP23K22409.


\section*{Data availability statements}
No new data were created or analysed in this study. 
%


\appendix

\renewcommand{\theequation}{\Alph{section}.\arabic{equation}}


\end{document}